\documentclass[12pt, leqno]{amsart}
\usepackage[all]{xy}
\usepackage{amsfonts,amsmath,oldgerm,amssymb,amscd}
\UseComputerModernTips

\newtheorem{theorem}{Theorem}

\newtheorem{lemma}[subsection]{{\bf Lemma}}
\newtheorem{coro}[subsection]{{\bf Corollary}}

\newcommand{\al}{\alpha}

\newcommand{\om}{\omega}
\newcommand{\lf}{\big{\lfloor}}
\newcommand{\rf}{\big{\rfloor}}
\newcommand{\del}{\delta}

\newcommand{\lc}{\big{\lceil}}
\newcommand{\rc}{\big{\rceil}}
\newcommand{\Z}{\mbox{$\mathbb Z$}}

\begin{document}
\title{Irreducibility of generalized Hermite-Laguerre Polynomials III} 
\author{Shanta Laishram}
\address{Stat-Math Unit, India Statistical Institute\\
7, S. J. S. Sansanwal Marg, New Delhi, 110016, India}
\email{shantalaishram@gmail.com}
\author[T. N. Shorey]{T. N. Shorey}
\address{Department of Mathematics\\
IIT Bombay, Powai, Mumbai 400076, India}
\email{shorey@math.iitb.ac.in}
\maketitle
\pagenumbering{arabic}
\pagestyle{headings}

\begin{abstract}
For a positive integer $n$ and a real number $\alpha$, the generalized Laguerre polynomials are defined by
\begin{align*}
L^{(\alpha)}_n(x)=\sum^n_{j=0}\frac{(n+\alpha)(n-1+\alpha)\cdots (j+1+\alpha)(-x)^j}{j!(n-j)!}.
\end{align*}
These orthogonal polynomials are solutions to \emph{Laguerre's Differential Equation} which arises in the
treatment of the harmonic oscillator in quantum mechanics.  Schur studied these Laguerre polynomials for its  interesting algebraic properties. He obtained irreducibility results of $L^{(\pm \frac{1}{2})}_n(x)$ and
 $L^{(\pm \frac{1}{2})}_n(x^2)$ and derived that the Hermite polynomials $H_{2n}(x)$ and
 $\frac{H_{2n+1}(x)}{x}$ are irreducible for each $n$.
In this article, we extend Schur's result by showing that the family of Laguerre polynomials
$L^{(q)}_n(x)$ and $L^{(q)}_n(x^d)$ with $q\in \{\pm \frac{1}{3}, \pm \frac{2}{3}, \pm \frac{1}{4}, \pm \frac{3}{4}\}$, where $d$ is the denominator of $q$, are irreducible for every $n$ except when
$q=\frac{1}{4}, n=2$ where we give the complete factorization. In fact, we derive it from a more general result.
\end{abstract}

\section{Introduction}

For a positive integer $n$ and a real number $\al$, the generalized Laguerre polynomials are
defined by
\begin{align*}
L^{(\al)}_n(x)=\sum^n_{j=0}\frac{(n+\al)(n-1+\al)\cdots (j+1+\al)(-x)^j}{j!(n-j)!}.
\end{align*}
Let $d>1$ be an integer and $q$ be a rational number with denominator equal to $d$ written in its reduced form
$$q=u+\frac{\al}{d}$$
where $u, \al\in \Z$ with $1\le \al <d$ and gcd$(\al, d)=1$. For integers $a_0, a_1, \cdots a_n$, let
\begin{align*}
G(x):=G_q(x)=&\sum^n_{j=0}a_j(n+q)(n-1+q)\cdots (j+1+q)d^{n-j}x^j\\
=&\sum^n_{j=0}a_jx^j\left(\prod^n_{i=j+1}(\al +(u+i)d)\right).
\end{align*}
This is an extension of Hermite polynomials and generalized Laguerre polynomials. In fact, when
$a_j=(-1)^j\binom{n}{j}$, we obtain $d^nn!L^{(q)}_n(\frac{x}{d})$ and Hermite polynomials are given by
\begin{align*}
H_{2n}(x)=(-1)^n2^{2n}n!L^{(-\frac{1}{2})}(x^2) \ {\rm and} \ H_{2n+1}(x)=(-1)^n2^{2n+1}n!xL^{(\frac{1}{2})}(x^2).
\end{align*}
Therefore we call $G(x)$ the generalized Hermite-Laguerre polynomial. We have
\begin{align*}
G(x^d):=G_q(x^d) ={\displaystyle{\sum_{j=0}^{dn}}} b_j x^j \ \
{\rm where} \ \ b_j =\begin{cases}
a_l{\displaystyle{\prod^n_{i=l+1}}} (\al+(u+i)d) \ & {\rm if} \ j = dl\\
0 \ & {\rm otherwise}.
\end{cases}
\end{align*}
We observe that the irreducibility of $G_q(x^d)$ implies the irreducibility of $G_q(x)$. There is
a slight difference in the notation of this paper from that of \cite{stirred}, \cite{GHL1} and \cite{GHL2};
$G_q(x)$ here is $G_{q+1}(x)$ in the above papers. The first result on the irreducibility of these
polynomials is due to Schur. Schur \cite{schur} proved
that $G_{-\frac{1}{2}}(x^2)$ with $a_n=\pm 1$ and $a_0=\pm 1$ are irreducible and this implies
the irreducibility of Hermite poynomial $H_{2n}$. Schur \cite{schur1} also established the irreducibility
of $\frac{H_{2n+1}(x)}{x}$ by showing that $G_{\frac{1}{2}}(x^2)$ with $a_n=\pm 1$ and $a_0=\pm 1$ is irreducible
except for $n=12$ where it may have a quadratic factor.
In this paper, we extend Schur's result by proving 

\begin{theorem}\label{Lag1/3}
Let $q\in \{\pm \frac{1}{3}, \pm \frac{2}{3}, \pm \frac{1}{4}, \pm \frac{3}{4}\}$. The Laguerre polynomials
$L^{(q)}_n(x)$ and $L^{(q)}_n(x^d)$, where $d$ is the denominator of $q$, are irreducible for every $n$
except when $q=\frac{1}{4}, n=2$ where
\begin{align*}
L^{(\frac{1}{4})}_2(x)=\frac{1}{32}(4x-3)(4x-15) \ {\rm and} \ L^{(\frac{1}{4})}_2(x^4)=\frac{1}{32}(4x^4-3)(4x^4-15).
\end{align*}
\end{theorem}

In fact we derive Theorem \ref{Lag1/3} from the following general result extending the theorems of
\cite{GHL1} and \cite{GHL2}. For a non-zero integer $m$, we denote by $P(m)$ the greatest prime
divisor of $m$ with the convention $P(\pm 1)=1$. Observe that if a polynomial of degree $m$ has a
factor of degree $k<m$, then it has a co-factor of degree $m-k$. Therefore when we consider a factor
of a polynomial of degree $m$, we always mean the factor whose degree is $\le \frac{m}{2}$.

\begin{theorem}\label{1/3}
Let $q\in \{\pm \frac{1}{3}, \pm \frac{2}{3}\}$. Assume that $P(a_0a_n)\le 3$ and further
$2\nmid a_0a_n$ if $\al+3(n+u)$ is a power of $2$. Then the polynomials $G(x)$ and $G(x^3)$
with $q\in \{-\frac{1}{3}, -\frac{2}{3}\}$ are both irreducible except when $q=-\frac{2}{3}, n=2$
where $G(x)$ may have a linear factor and $G(x^3)$ may have a cubic factor or when
$q=-\frac{1}{3}, n=43$ where $G(x^3)$ may have a factor of degree $5$. Further the polynomials
$G(x)$ and $G(x^3)$ with $q\in \{\frac{1}{3}, \frac{2}{3}\}$ are both irreducible except possibly
when
\begin{itemize}
\item[$(i)$]  $1+3n=2^a$ where $G_{\frac{1}{3}}(x)$ may have a linear factor and
$G_{\frac{1}{3}}(x^3)$ may  have a quadratic or a cubic factor.
\item[$(ii)$]  $2+3n=2^a$ and $n\neq 42$ where $G_{\frac{2}{3}}(x^3)$ may  have a quadratic factor.
\item[$(iii)$]  $2+3n=2^b5^c, b\ge 0, c>0$ where $G_{\frac{2}{3}}(x)$ may have a linear factor and
$G_{\frac{2}{3}}(x^3)$ may  have a cubic factor.
\item[$(iv)$]  $n=42$ where $G_{\frac{2}{3}}(x)$ may have a quadratic factor and
$G_{\frac{2}{3}}(x^3)$ may  have a factor of degree in $\{2, 4, 5, 6\}$.
\end{itemize}
\end{theorem}

\begin{theorem}\label{1/4}
Let $q\in \{\pm \frac{1}{4}, \pm \frac{3}{4}\}$. Assume that $P(a_0a_n)\le 3$ and further $P(a_0a_n)\le 2$ if
$\al+4(n+u)$ is a power of $3$ when $q\in \{-\frac{1}{4}, -\frac{3}{4}\}$ and
$3|(\al+4n)$ when $q\in \{\frac{1}{4}, \frac{3}{4}\}$. Then the polynomials $G_{-\frac{3}{4}}(x)$ and
$G_{-\frac{3}{4}}(x^4)$ are both irreducible. Further
$G_{\pm \frac{1}{4}}(x), G_{\pm \frac{1}{4}}(x^4), G_{\frac{3}{4}}(x)$ and
$G_{\frac{3}{4}}(x^4)$ are irreducible except possibly when $3+4(n-1)=3^a$ if $q=-\frac{1}{4}$;
$1+4n=3^b5^c, b, c\ge 0, b+c>0$ if $q=\frac{1}{4}$ and $3+4n=7^y$ if $q=\frac{3}{4}$ where $G_q(x)$ may
have a linear factor and $G_q(x^4)$ may have a factor of degree $4$.
\end{theorem}

It follows from Theorem \ref{1/4} that if $n$ is a multiple of $3$, then $G_q(x^4)$ is irreducible for
$q\in \{\pm \frac{1}{4}, \pm \frac{3}{4}\}$. In Theorem \ref{1/3}, the case $q=-\frac{2}{3}, n=2$ is necessary since $G_q(x)=(x+2)^2$ and $G_q(x)=(x^3+2)^2$ when $a_0=a_1=a_2=1$.  The assumptions on $a_0a_n$ in Theorems \ref{1/3} and \ref{1/4} are satisfied if $|a_0|=|a_n|=1$; in fact the assumptions of Theorem \ref{1/4} are satisfied if $P(a_0a_n)\le 2$.  Therefore the assertions of Theorems are valid whenever $|a_0|=|a_n|=1$ and further for Theorem \ref{1/4} whenever $P(a_0a_n)\le 2$. 
We believe that for suitable choices of $a_j$'s, many of the polynomials $G_q(x)$ with conditions given in Theorems \ref{1/3} and \ref{1/4} will have linear factor or $G(x^d)$ will have a factor of degree $\leq d$ but we have not found out examples for the same. It will be interesting to either give such examples or prove irreducibility completely for those cases. 

The proofs of Theorems \ref{1/3} and \ref{1/4} are given in Sections $5-7$.
Further we prove Theorem \ref{Lag1/3} in Section $8$. The following result used in
the proof of Theorem \ref{1/4} is also of independent interest.
\begin{theorem}\label{>4k+1}
Let $k\geq 2, n>4k$ and $2\nmid n$. Then
\begin{align}\label{P>4k+1}
P(n(n+4)\cdots (n+4(k-1)))>4(k+1)
\end{align}
unless $k=2, n\in \{11, 21, 45, 77, 121\}$ and $k=3, n=117$.
\end{theorem}
As an immediate consequence of Theorem \ref{>4k+1}, we obtain
\begin{coro}\label{>4k}
Let $k\geq 2, n>4k$ and $2\nmid n$. Then
\begin{align}\label{P>4k}
P(n(n+4)\cdots (n+4(k-1)))>4k
\end{align}
unless $k=2, n\in \{21, 45\}$.
\end{coro}
We give a proof of Theorem \ref{>4k+1} in Section $4$. In Section $2$, we give some preliminaries and
in Section $3$, we give statements and results on Newton polygons.

The proof of Theorems \ref{Lag1/3}-\ref{1/4} involve combinations of ideas of $p-$adic Newton polygons
with estimates on the greatest prime factor of a product of consecutive terms of an arithmetic progression.
The new ingredients in the paper are Theorem \ref{>4k+1} and the exploitation of arithmetic properties of
some special numbers arising out of application of Newton polygon ideas  and extending the arguments for
$G_q(x)$ to $G_q(x^d)$ where $d$ is the denominator of $q$.

\section{Preliminaries}

For positive integers $m, d, k$, we write
\begin{align*}
\Delta(m, d, k)=m(m+d)\cdots (m+d(k-1)).
\end{align*}
Recall that for an integer $m>1$, we denote by $P(m)$ the greatest prime factor of $m$ and we put
$P(1)=1$. The following result is \cite[Theorem 3]{GHL1}.

\begin{lemma}\label{d3}
Let $k\ge 2$ and $d=3$. Let $m$ and $k$ be positive integers such that $3\nmid m$ and
$m>3k$. Then
\begin{align}\label{d>3k}
P(\Delta(m, 3, k))>3k  \ \ {\rm unless} \ \ (m, k)=(125, 2).
\end{align}
\end{lemma}

For a prime $p$ and a nonzero integer $r$, we define $\nu(r) = \nu_p(r)$
to be the nonnegative integer such that $p^{\nu(r)}|r$ and $p^{\nu(r)+1}\nmid r$.
We define $\nu(0) = +\infty$.
The following classical result is due to Legendre. See for example,
Hasse \cite[Ch. 17, no. 3, p. 263]{hasse}.
\begin{lemma}\label{m!}
Let $p$ be a prime. For any integer $m\ge 1$, write $m$ in base $p$ as
\begin{align*}
m=m_tp^t+m_{t-1}p^{t-1}+\cdots +m_1p+m_0
\end{align*}
where $0\le m_i\le p-1$ for $0\le i\le t$. Then
\begin{align*}
\nu_p(m!)=\frac{m-s_p(m)}{p-1}
\end{align*}
where $s_p(m)=m_t+m_{t-1}+\cdots +m_1+m_0$ is the sum of digits of $m$ in base $p$. In
particular $\nu_p(m!)\le \frac{m-1}{p-1}$ since $s_p(m)\ge 1$.
\end{lemma}

The next lemma is on solutions of some equations.

\begin{lemma}\label{nag}
Let $x>0, y>0, z>0$ be integers. The solutions of the following equations are given by
\begin{center}
\begin{tabular}{|c|c|c|} \hline
& Equation & Solutions \\ \hline
$(i)$ & $a^x-b^y=\pm 1,  a, b \in \{2, 3, 5\}$ & $3-2=1, 2^2-3=1, 5-2^2=1, 3^2-2^3=1$ \\ \hline
$(ii)$ & $2^x+3^y=5^z$ & $2+3=5, 2^4+3^2=5^2$ \\ \hline
$(iii)$ & $2^x+3^y=7^z$ & $2^2+3=7$ \\ \hline
$(iv)$ & $2^x3^y-5^z=\pm 1$ & $2\cdot 3-5=1, 2^3\cdot 3-5^2=-1$ \\ \hline
$(v)$ & $3^x5^y-2^z=\pm 1$ & $3\cdot 5-2^4=-1$ \\ \hline
$(vi)$ & $2^x5^y-3^z=\pm 1$ & $2\cdot 5-3^2=1, 2^4\cdot 5-3^4=-1$ \\ \hline
\end{tabular}
\end{center}
\end{lemma}

The assertion $(i)$ is a special case of Catalan's Conjecture, now Mihailescu's Theorem when 
$x>1, y>1$, see \cite{mihai}. The case  $x=1$ or $y=1$ is immediate. The assertions $(ii)$ and $(iii)$ are due to Nagell \cite{nagell}.
For assertions $(iv)-(vi)$, see  \cite[Lemma 4]{LS>2k}.

The next lemma is \cite[Corollary 2.12]{GHL1} together with computations for $X\le 80$.

\begin{lemma}\label{upto7}
Let $X\ge 1, 3\nmid X$ and $1\le i\le 7$. Then the solutions of
\begin{align*}
P(X(X+3i))=5 \ \ {\rm and} \ 2|X(X+3i)
\end{align*}
are given by
\begin{align*}
(i, X)\in \{&(1, 2), (1, 5), (1, 125), (2, 4), (2, 10), (2, 250), (3, 1), (3, 16),
(4, 8), \\
&(4, 20), (4, 500), (5, 5), (5, 10), (5, 25), (5, 625), (6, 2), (6, 32), (7, 4)\}.
\end{align*}
\end{lemma}

We also need the following result which is \cite[Corollary 2.3]{GHL1} and \cite[Corollary 4.3]{GHL2}.

\begin{lemma}\label{4.3}
Let $d\in \{3, 4\}$, $\gcd(n, d)=1$ and $6450<n\le 10.6\cdot 3k$ if $d=3$ and
$10^6<n\le 138\cdot 4k$ if $d=4$. Then $P(\Delta(n, d, k))\ge n$.
\end{lemma}

Let $p_{i, \mu, l}$ denote the $i$th prime congruent to $l$ modulo $\mu$. Let
$\del_{\mu}(i, l)=p_{i+1, \mu, l}-p_{i, \mu, l}$. The following lemma is a
computational result.

\begin{lemma}\label{diff}
$(i)$ Let $l\in \{1, 2\}$. Then $\del_3(i, l)\le 60$ for $p_{i, 3, l}\le 7348$.

\noindent
$(ii)$ Let $l\in \{1, 3\}$. Then $\del_4(i, l)\le 264$ for $p_{i, 4, l}\le 1.1\cdot 10^7$
except when $(p_{i, 4, l}, p_{i+1, 4, l})\in \{(7856441, 7856713), (10087201, 10087481),
(3358151, 3358423)$,\\ $(5927759, 5928031), (9287659, 9287939)\}$.
\end{lemma}

\section{Newton Polygons}

Let $f(x) =\sum^m_{j=0} a_jx^j\in \Z[x]$ with $a_0a_m\neq 0$ and $p$ be a prime.
Let $S$ be the following set of points in the extended plane:
\begin{align*}
S = \{(0, \nu(a_m)), (1, \nu(a_{m-1})), (2, \nu(a_{m-2})), \cdots,
(m-1, \nu(a_1)),  (m, \nu(a_0))\}.
\end{align*}
Consider the lower edges along the convex hull of these points. The
left-most endpoint is $(0, \nu(a_m))$ and the right-most endpoint is
$(m, \nu(a_0))$. The endpoints of each edge belong to S and the slopes of
the edges increase from left to right. When referring to the edges of a
Newton polygon, we shall not allow two different edges to have the same
slope. The polygonal path formed by these edges is called the Newton
polygon of $f(x)$ with respect to the prime $p$ and we denote it by $NP_p(f)$.
The end points of the edges on $NP_p(f)$ are called the \emph{vertices} of $NP_p(f)$. We call 
the $x-$axis of the vertices to be \emph{breaks} of the Newton polygon and usually write 
$0=:x_0<x_1<\cdots <x_s:=m$ as the breaks where $(x_i, \nu(a_{m-x_i}), 0\le i\le s$ are the 
vertices of $NP_p(f)$.  We define the \emph{Newton function} of $f$ with respect to the prime $p$ as the real
function $f_p(x)$ on the interval $[0, m]$ which has the polygonal path formed
by these edges as its graph. Hence $f_p(i) = \nu(a_{m-i})$ for $i=0, m$ and at
all points $i$ such that $(i, \nu(a_{m-i}))$ is a vertex of $NP_p(f)$. We need the
following result which is a refinement of a lemma due to Filaseta \cite[Lemma 2]{Filbes}.
 This was proved in \cite[Lemma 2.13]{stirred}.

\begin{lemma}\label{<r}
Let $k, m$ and $r$ be integers with $m\ge 2k>0$.  Let
$g(x)=\sum^m_{j=0}b_jx^j\in \Z[x]$ and let $p$ be a prime such that
$p\nmid b_m$. Denote the Newton function of $g(x)$ with respect to $p$ by $g_p(x)$.
Let $a_0, a_1, \ldots, a_m$  be integers with $p\nmid a_0a_m$. Put
$f(x)=\sum^m_{j=0}a_jb_jx^j\in \Z[x]$. If $g_p(k)>r$ and $g_p(m)-g_p(m-k)<r+1$, then
$f(x)$ cannot have a factor of degree $k$.
\end{lemma}

Lemma \ref{<r} implies the following result of Filaseta \cite[Lemma 2]{Filbes} 
together with a remark just after its proof in \cite{Filbes}. 

\begin{coro}\label{1/k}
Let $l, k, m$ be integers with $m\ge 2k>2l\ge 0$. Suppose $g(x)=\sum^m_{j=0}b_jx^j\in \Z[x]$ and
$p$ be a prime such that $p\nmid b_m$ and $p|b_j$ for $0\le j\le m-l-1$ and the right most
edge of the $NP_p(g)$ has slope $<\frac{1}{k}$. Then for any integers
$a_0, a_1, \ldots, a_m$ with $p\nmid a_0a_m$, the polynomial $f(x)=\sum^m_{j=0}a_jb_jx^j$
cannot have a factor with degree in $[l+1, k]$.
\end{coro}

\begin{proof}
Since $p|b_j$ for $0\leq j \leq m-l-1$, we have $g_p(K)>0$ for $K \in [l+1,k]$. Let
$(m_1,g_p(m_1))$ be the starting point of the rightmost edge of $NP_p(g)$. Then
\begin{align*}
\frac{1}{m-m_1} \leq \frac{g_p(m)-g_p(m_1)}{m-m_1}< \frac{1}{k}
\end{align*}
 giving $m_1 <m-k\leq m-K$ for $K \leq k$. Hence for $K\in[l+1,k]$, $(m-K,g_p(m-K))$ lie on the
 rightmost edge implying $\frac{g_p(m)-g_p(m-K)}{K}< \frac{1}{k} \leq \frac{1}{K}$. Thus
 $g_p(m)-g_p(m-K) <1$. Now we apply Lemma \ref{<r} with $r=0$ to get the assertion.
\end{proof}

Unless otherwise mentioned, we always take $l=k-1$ while using 
Corollary \ref{1/k}. Next we need the following result generalizing 
\cite[Lemma 1]{GHL2} where the case $u=-1$ was proved. 

\begin{lemma}\label{mainupdate}
Let $u\in \{-1, 0\}$ and $1\le k\le \frac{n}{2}$. Suppose there is a prime $p$ satisfying
\begin{align}\label{condp}
p>d, p>\min(2k, d(d-1)) \ {\rm and \ further} \ p\ge \frac{(k+.5)d}{d-1} \ {\rm if} \ u=-1, p\le 2k
\end{align}
and
\begin{align*}
p|\prod^{k-1}_{j=0}(\al +(u+n-j)d), \ \ p\nmid \prod^{k}_{j=1}(\al +(u+j)d), \ p\nmid a_0a_n.
\end{align*}
Then $G(x)$ has no factor of degree $k$ and $G(x^d)$ does not have a factor of degree in
$[dk-d+1, dk]$. Further for $n$ odd and $k=\lf\frac{n}{2}\rf$, $G(x^d)$ does not have a factor of degree in
$[dk+1, dk+\frac{d}{2}]$.
\end{lemma}

\begin{proof}
We use Corollary \ref{1/k}. 
 We take $(m, k, l)$ to be $(n, k, k-1)$ for $G(x)$ having a factor of degree $k$ and $(dn, dk, d(k-1))$ for $G(x^d)$ having a factor of degree in $[dk-d+1,dk]$. Further for
 $n$ odd and $G(x^d)$ having a factor of degree in $(dn_0, dn_0+\frac{d}{2}]$ where $n_0=\lf\frac{n}{2}\rf$,
 we take $(m, k, l)$ to be $(dn, dn_0+\lf\frac{d}{2}\rf, dn_0)$.  We observe that the assumptions of
 Corollary \ref{1/k} are satisfied.  Let
\begin{align*}
\Delta_j=(\al +(u+1)d)\cdots (\al +(u+j)d).
\end{align*}
By Corollary \ref{1/k}, it suffices to show that
\begin{align}\label{phij}
\frac{\nu_p(\Delta_j)}{j}<\frac{1}{k+\frac{1}{2}} \ \ {\rm for} \ 1\le j\le n.
\end{align}
Let $j_0\ge 1$ be the minimum $j$ such that $p|(\al +(u+j)d )$ and we write $\al +(u+j_0)d=pl_0$.
Then $j_0>k$ since $p\nmid \Delta_k$. Note that $j_0\le p$. Further $1\le l_0<d$ otherwise
$l_0\ge d+1$ and $p\le pl_0-pd=\al +(u+j_0-p)d\le \al+ud<d<p$, a contradiction. Also
$p(d-1)\ge pl_0=\al +(u+j_0)d\ge \al+(u+k+1)d$. Thus $p(d-1)>(k+1)d$ if $u=0$. If $u=-1$ and $p>2k$, 
we have $p(d-1)\ge (2k+1)(d-1)\ge (k+.5)d$ since $d>1$.  This together with \eqref{condp}  imply
\begin{align}\label{.5d}
p\ge \frac{(k+.5)d}{d-1}.
\end{align}
For showing \eqref{phij}, we may restrict to those $j$ such that $\al +(u+j)d =pl$ for some $l$. Then
$(j-j_0)d =p(l-l_0)$ implying $d|(l-l_0)$ since $\gcd(p, d)=1$. Writing $l=l_0+sd$, we get
$j=j_0+ps$. Note that if $p|(\al +(u+i)d )$, then $\al +(u+i)d =p(l_0+r d)$
for some $r\ge 0$. Hence we have
\begin{align*}
\nu_p(\Delta _j)&=\nu_p((pl_0)(p(l_0+d))\cdots (p(l_0+sd ))=
s+1+\nu_p(l_0(l_0+d)\cdots (l_0+sd))
\end{align*}
for some integer $s\ge 0$. Further we may suppose that $s>0$ otherwise the assertion
follows since $p>d>l_0$ and $j_0>k$.  Further from \eqref{phij}, $j=j_0+ps\geq k+1+ps$ and
$\frac{k+1+ps}{k+.5}=1+\frac{ps+.5}{k+.5}$, it suffices to show
\begin{align}\label{l_or}
\phi_s:=s+\nu_p(l_0(l_0+d)\cdots (l_0+sd))<\frac{ps+.5}{k+.5}.
\end{align}
We consider two cases.

\noindent
{\bf Case I:} Assume that $s<p$. Then $p$ divides at most one term of
$\{l_0+id : 0\le i\le s\}$ and we obtain from $l_0+sd<(s+1)d<p^2$ that $\phi_s\le s+1$. To show
\eqref{l_or}, we need to show that $\frac{ps+.5}{k+.5}-1>s$ or
$s(p-k-\frac{1}{2})>k$. This is true if $p\ge 2k+1$. Thus we may suppose that $p\le 2k$.  Since
 $p\geq \frac{(k+.5)d}{d-1}$ by \eqref{.5d}, we get $(d-1)(p-k-.5)>k$. Thus
$s(p-k-.5)>k$ is valid for $s\ge d-1$ and therefore we may now assume $s\le d-2$.
Then $l_0+sd\le d-1+(d-2)d<p$ and hence $\phi_s=s$. Now \eqref{l_or} is valid since $p\ge k+1$.

\noindent
{\bf Case II:} Let $s\ge p$. Let $r_0\le s$ be such that $\nu_p(l_0+r_0d)$ is maximal. Then
\begin{align*}
\phi_s\le s+\nu_p(l_0+r_0d)+ \nu_p(r_0!(s-r_0)!)\le s+
\frac{\log (l_0+sd)}{\log p}+\frac{s-1}{p-1}
\end{align*}
by using Lemma \ref{m!}.  We have $p\ge d+1$. This with $l_0\le d-1<p\le s$ imply
$\log (l_0+sd)\le \log s(d+1)=\log s +\log (d+1)\le \log s +\log p$. Hence
\begin{align*}
\phi_s\le s+\frac{s}{p-1}+\frac{\log s}{\log p}+1-\frac{1}{p-1}.
\end{align*}
To show \eqref{l_or}, it is enough to show that
\begin{align*}
1+\frac{1}{p-1}+\frac{\log s}{s\log p}+\frac{1}{s}(1-\frac{1}{p-1}-\frac{1}{2k+1})<\frac{2p}{2k+1}.
\end{align*}
The left hand side of the above inequality is a decreasing function in $s$.  Since $s\ge p$,
the left hand side of the above inequality is at most
\begin{align*}
1+\frac{1}{p-1}+\frac{1}{p}+\frac{1}{p}(1-\frac{1}{p-1}-\frac{1}{2k+1})=1+\frac{3}{p}-\frac{1}{p(2k+1)}
\end{align*}
 and therefore it suffices to show
\begin{align}\label{iterm2}
1+\frac{3}{p}-\frac{1}{p(2k+1)}<\frac{2p}{2k+1}.
\end{align}
Let $p\ge 2k+1$. Then $p\ge 3$ and the left hand side of \eqref{iterm2} is at most
\begin{align*}
1+1-\frac{1}{p(2k+1)}<2\le \frac{2p}{2k+1}.
\end{align*}
Thus we may assume that $p\le 2k$. Then $p>d(d-1)$. Further $d\ge 3$ since
$p(d-1)\ge \al+(u+k+1)d$ and $p<2k$. Therefore the left hand side of \eqref{iterm2} is at most
\begin{align*}
1+\frac{3}{d(d-1)}-\frac{1}{p(2k+1)}<1+\frac{1}{d-1}=\frac{d}{d-1}\le \frac{2p}{2k+1}.
\end{align*}
by  \eqref{.5d}.
\end{proof}

The following corollary easily follows from Lemma \ref{mainupdate}.
\begin{coro}\label{dlk}
Let $u\in \{0, -1\}$ and $n\ge 2k>0$. Suppose that $P(a_0a_n)\le d$ and
\begin{align*}
P ((\alpha+d(u+n-k+1)) \cdots (\alpha + d(u+n)))>d(u+k+1).
\end{align*}
Then $G_q(x)$ does not have a factor of degree $k$ and $G_q(x^d)$
do not have a factor of degree in $[dk-d+1, dk]$. Further for $n$ odd and 
$k=\lf\frac{n}{2}\rf$, $G_q(x^d)$ does not have a factor of degree in $[dk+1, dk+\frac{d}{2}]$.
\end{coro}

\section{Proof of Theorem \ref{>4k+1}}

Let $k\ge 2, n>4k$ and $2\nmid n$. Assume that $P(n(n+4)\cdots (n+4(k-1)))\le 4(k+1)$. Let
\begin{align*}
S_M=\{m: m\ge1, m \ {\rm odd}, P(m(m+4))\le M\}.
\end{align*}
The set $S_M$ for $M\le 31$ is given in \cite{lehmer} and for $M=100$ in \cite{najman}. In fact,
$m=x-2$ with $x$ listed in the table \cite{najman} and $m=N-4$ for $N$ listed in \cite[Table IIIA]{lehmer}.

Let $k=2$. Then $P(n(n+4))\le 11$ implying $n\in S_{11}$. Since $n>8$, we have $n\in \{11, 21, 45, 77, 121\}$.

Let $k=3$. Then $P(n(n+4)(n+8))\le 13$ giving $P(n(n+4))\le 13$ and $P((n+4)(n+8))\le 13$. Hence both
$n\in S_{13}$ and $n+4\in S_{13}$. Since $n>12$, we have $n=117$.

Let $4\le k \le 8$. Since $P(\Delta(n, 4, k))\le 4k+4$, we have $P(n(n+4))\le 31$, $P((n+4)(n+8))\le 31$ and
$P((n+8)(n+12))\le 31$.  Hence $n+4i\in S_{31}$ for each $0\le i\le 2$.  Then $n\in \{17, 19, 21, 23, 27, 87\}$. For these values $n$ and $k$ such that $n>4k$, we check that $P(\Delta(n, 4, k))>4(k+1)$. Thus $k\ge 9$.

Let $9\le k<67$. Since $P(\Delta (n, 4, k))\le 4k+4$, we have $\omega(\Delta(n, 4, k))\le \pi(4k+4)$.
We check that $k-\pi(4k+4)+\pi(100)>\lc \frac{k}{2}\rc$. Hence there is some
$i_0$ with $0\le i_0\le k-2$ such that $P((n+4i_0)(n+4(i_0+1)))\le 100$. Then $n+4i_0=m\in S_{100}$.
Suppose $m>10^7$. We check that $P(\prod^4_{i=1}(m-4i))>280$ and $P(\prod^4_{i=1}(m+4+4i))>280$ for each $m\in S_{100}$ and $m>10^7$. Thus $P(\prod^{k-1}_{i=0}(n+4i)>280$ implying the assertion when $n+4i_0>10^7$. Thus we can assume that $m\le 10^7$. Then $n\le n+4i_0\le 10^7$. We compute that $P(\prod^8_{i=0}(n+4i))>280$ except when $n\in \{465, 469, 473, 885, 1513\}$. For these values of $n$, we see that $P(\displaystyle{\prod^8_{i=0}}(n+4i))>52$ which is $>4(k+1)$ for $9\le k\le 12$. Further
for these values of $n$, we also have $P(\prod^{12}_{i=0}(n+4i))>280$ which is $>4(k+1)$ for $13\le k<67$.

Thus we may suppose that $k\ge 67$. Since $P(\Delta(n, 4, k))\leq 4k+4<n+4$, we see that 
each of $n+4, n+8, \cdots, n+4(k-1)$ are composite and hence there is a prime 
$p_{i, 4, l}\equiv n($mod $4)$ such that $p_{i, 4, l}\le n<n+4<n+4(k-1)<n+4k\le p_{i+1, 4, l}$. Thus 
$p_{i+1, 4, l}-p_{i, 4, l}\ge 4k$. Let $n\le 1.1\cdot 10^7$. By Lemma \ref{diff}, we can assume that $k\in \{67, 68, 69, 70\}$ and $p_{i, 4, l}\le n<n+4(k-1)<n+4k\le p_{i+1, 4, l}$ for $(p_{i, 4, l}, p_{i+1, 4, l})$ 
listed in Lemma \ref{diff}. For such values of $n$, we check that that $P(\prod^{k}_{i=0}(n+4i))>284$. Hence we can assume that $n>1.1\cdot 10^7$.

Let $4k<n<4k+4$. Since $n>1.1\cdot 10^7$, we have $10^6<n+4\leq 138\cdot 4(k-1)$. By 
Lemma \ref{4.3},  we have $P(\Delta(n+4, 4, k-1) )\geq n+4$. Hence 
$P(\Delta(n, 4, k) )\geq P(\Delta(n+4, 4, k-1) )\geq n+4>4k+4$. Thus we can assume that 
$n>4k+4$. Further again by Lemma \ref{4.3}, we can now assume that $n>138\cdot 4k$.

Since $P(\Delta(n, 4, k))\le 4k+4$, we have $\om(\Delta (n, 4, k))\le \pi(4k+4)-1$. We continue as in 
\cite[Section 3]{GHL2} with $d=4, t=\pi(4k+4)-1$ to obtain
\begin{align}\label{Lbd}
n\le \left((k-1)! \prod_{p\le p_l}p^{L_0(p)}\right)^{\frac{1}{k+1-\pi(4k+4)}}
\end{align}
for every $l\ge 1$ where
\begin{align*}
L_0(p)=\begin{cases} \min (0, h_p(k+1-\pi(4k))-\sum^{h_p}_{u=1}\lf \frac{k-1}{p^u}\rf ) & {\rm if} \ p\nmid d\\
-\nu_p((k-1)!) & {\rm if} \ p| d
\end{cases}
\end{align*}
with $h_p\ge 0$ such that $[\frac{k-1}{p^{h_p+1}}]\le k+1-\pi(4k+4)<[\frac{k-1}{p^{h_p}}]$. 
Taking $l=3$ in \eqref{Lbd}, we find that $n<1.1\cdot 10^7$ when $k\le 400$. Thus $k>400$.

We now write $n=v\cdot 4k$ with a real number $v\ge v_0:=138$. We continue as in the last paragraph of \cite[pp. 433]{GHL2} to obtain
\begin{align*}
\log (v_0\cdot 8\cdot e)<\frac{4\log (v_0\cdot 4k)}{\log (4k+3)}
\left(1+\frac{1.2762}{\log (4k+3)}\right).
\end{align*}
The right hand side of the above inequality is a decreasing function of $k$ and the inequality does not hold at $k=401$. This is a contradiction.
\qed

\section{Proof of $G_{u+\frac{\al}{3}}(x^3)$ not having a factor of degree $\ge 4$}

Let $d=3$, $\al \in \{1, 2\}$, $u\in \{0, -1\}$ and $P(a_0a_n)\leq 3$. It suffices to show
$G_{u+\frac{\al}{3}}(x^3)$ does not have a factor of degree in $\{3k, 3k-1, 3k-2\}$ for
$2\le k\le \frac{n}{2}$ and further a factor of degree $\frac{3(n-1)}{2} +1$ when $n$ is odd. By 
Corollary \ref{dlk}, we may assume that $P(\prod^{k-1}_{i=0}(\al +3(u+n-i)))<3(u+k+1)$. 
Since $n\ge 2k$, by Lemma \ref{d3}, we have $u=0$ and 
\begin{align}\label{3k+al}
3k<P(\prod^{k-1}_{j=0}(\al +3(n-j)))<3(k+1)
\end{align}
except when $k=2$ and $\al+3(u+n-k+1)=125$.

Let $k=2$ and $\al+3(u+n-k+1)=125$. Then $\al=2$ and $(u, n)\in \{(-1, 43), (0, 42)\}$. 
We consider the Newton polygon with respect to $p=2$ of the polynomials $G_{u+\frac{2}{3}}(x^3)$ with 
all $a_j'$s equal to $1$. The breaks of the Newton polygon are 
$0<32\cdot 3<40\cdot 3<43\cdot 3=3n$ when $u=-1, n=43$ and
$0<32\cdot 3<40\cdot 3<42\cdot 3=3n$ when $u=0, n=42$. Further the minimum slope(slope 
of the left most edge) is $\frac{1}{3}(1+\frac{1}{32})$ and the maximum slopes (slope of the 
right most edge) are $\frac{4}{9}$ and $\frac{1}{2}$ when $(u, n)=(-1, 43), (0, 42)$,  respectively. Thus by
Lemma \ref{<r} with $r=\lf \frac{t}{3}\rf$, $t\in \{4, 5, 6\}$, the polynomials
$G_{-1+\frac{2}{3}}(x^3)$ does not have factor of degree $t\in \{4, 6\}$. Hence $G_{-1+\frac{2}{3}}(x^3)$
may have a factor of degree $5$ when $n=43$ and $G_{\frac{2}{3}}(x^3)$ may have factor of degree
$t\in \{4, 5, 6\}$ when $n=42$.

Therefore we now suppose that $\al +3(u+n-k+1)\neq 125$ when $k=2$. By Lemma \ref{mainupdate},
we may restrict to those $k$ such that $P(\prod^{k-1}_{j=0}(\al +3(n-j)))=\al+3k$. Thus
$\al=1$ if $k$ is even and $\al=2$ if $k$ is odd. Let
\begin{align*}
R(k)=\{p: p|\prod^k_{i=1}(\al+3i), p \ {\rm prime}\}
\end{align*}
where $\al=1$ if $k$ is even and $\al=2$ if $k$ is odd. Again by Lemma \ref{mainupdate}, we
may suppose that $p|\prod^{k-1}_{j=0}(\al +3(n-j))$ imply $p\in R(k)$. Thus
$\om(\prod^{k-1}_{j=0}(\al +3(n-j)))\le |R(k)|$. Since $\al+3k$ is prime, we now have 
\begin{align*}
|R(k)|= \begin{cases}
\pi_1(3k+1)+\pi_2(\frac{3k+1}{2})=\pi_1(3k)+1+\pi_2(\frac{3k}{2}) & {\rm if} \ k \ {\rm is \ even}\\
\pi_2(3k+2)+\pi_1(\frac{3k+2}{2})=\pi_2(3k)+1+\pi_1(\frac{3k+1}{2}) & {\rm if} \ k \ {\rm is \ odd}\\
\end{cases}
\end{align*}
where $\pi_l(x)=|\{p\le x: p\equiv l({\rm mod} \ 3)\}|$ for $l\in \{1, 2\}$.

Let $k=2$. Then $p|(1+3n)(1+3n-3)$ imply  $p\in \{2, 7\}$. Hence $\{1+3n, 1+3n-3\}=\{2^a, 7^b\}$ 
for some positive integers $a, b$. Hence $7^b-2^a=\pm 3$. If $a\ge 3$, we get a contradiction 
modulo $8$. Hence $a\le 2$ and we have the only solution $7-4=3$. Therefore 
$1+3n=7, 1+3n-3=4$ giving $n=2$. This is not possible since $n\ge 2k$. 

Thus $k\ge 3$. Let $k\geq 20$. Let $l\in \{1, 2\}$ and $m$ is congruent to $l$ modulo $3$. Then note that if the set $\{m, m+3, \cdots , m+3(k-1)\}$ does not 
contain a prime, then the difference between two consecutive primes 
congruent $l($mod $3)$ is at least $(m+3k)-(m-3)=3k+3\geq 63$ contradicting 
Lemma \ref{diff} $(i)$ if $m\leq 7348$. Therefore the set $\{m, m+3, \cdots, m+3(k-1)\}$ contains a prime if $m\le 7348$ and hence $P(\prod^{k-1}_{i=0}(m+3i))\ge m$ if $m\le 7348$. For $3\le k<20$, we check that $P(\prod^{k-1}_{i=0}(m+3i))\ge \min(m, 3(k+1))$ for $3k<m\le 7348, 3\nmid m$ except when $k=3, m=22$. 
Thus for $k\ge 3$, we may assume by \eqref{3k+al} that either $\al+3(n-k+1)>7348$ or $k=3, \al+3(n-k+1)=22$. 
Since $\al=2$ when $k$ is odd, we obtain $\al+3(n-k+1)>7348$. Let $3\le k\le 8$. After deleting terms in 
$\{\al+3n, \al+3(n-1), \cdots, \al+3(n-k+1)\}$ divisible by $p\in R(k), p\ge 7$ we are left with at least $2$ indices $0\le i_1<i_2\le 7$ such that $p|(\al+3(n-i_1))(\al+3(n-i_2))$ imply $p\in \{2, 5\}$. By putting $X=\al+3(n-i_2)$, we obtain from
Lemma \ref{upto7} that $X\le 625$. But $X=\al+3(n-i_2)\ge \al+3(n-k+1)>7348$ which is a contradiction.

Thus we now have $k\ge 9$ and $\al+3(n-k+1)>7348$. Further we may also assume that
$\al+3(n-k+1)\ge 10.6\cdot 3k$ by Lemma \ref{4.3} and \eqref{3k+al}. By taking 
$m=\al+3(n-k+1), t=|R(k)|$ in 
\cite[(4)]{GHL2}, we obtain from \cite[(6)]{GHL2} that $\al+3(n-k+1)<4480$ for $9\le k\le 180$.
Thus we may suppose that $k>180$. We proved in the last para of
\cite[Section 3(A), pp. 62]{GHL1} that $\om(\prod^{k-1}_{i=0}(m+3i))\ge \pi(3k)$ for $k>180$
when $m>3k$ and $3\nmid m$. Therefore $\om(\prod^{k-1}_{j=0}(\al +3(n-j)))\ge \pi(3k)$. But
$\pi(3k)=\pi_1(3k)+\pi_2(3k)+1$ and we will have the contradiction $\pi(3k)>|R(k)|$ if 
\begin{align}\label{3k/2p} 
\begin{split}
\pi_2(3k)>\pi_2(3k/2) \ & {\rm if} \ k \ {\rm is \ even}\\
\pi_1(3k)>\pi_1((3k+1)/2) \ & {\rm if} \ k \ {\rm is \ odd}.  
\end{split}
\end{align}
We check that it is true when $3k/2\le 6450$. Hence we now assume $3k/2>6450$. Taking 
$(m, k_1)=(3k/2+1, k/2)$ if $k$ is even and $(m, k_1)=((3k+1)/2+3, (k-1)/2)$ if $k$ is odd, we see from 
Lemma \ref{4.3} that $P:=P(\Delta(m, 3, k_1)\geq m$. We note that $m\equiv 1, 2$ modulo $3$ according as 
$k$ is even or $k$ is odd, respectively.  Further observe that $2P\ge 2m>m+3(k_1-1)$ and hence $P$ is 
one of the terms of $m, m+3, \cdots, m+3(k_1-1)$ giving the assertion \eqref{3k/2p}. 
\qed

\section{Proof of $G_{u+\frac{\al}{4}}(x^4)$ not having a factor of degree $\ge 5$}

Let $d=4, u\in  \{0, -1\}$  and $\al \in \{1, 3\}$. It suffices to show 
that $G_{u+\frac{\al}{4}}(x^4)$ does not have a factor of degree in 
$\{4k, 4k-1, 4k-2, 4k-3\}$ for $2\le k\le \frac{n}{2}$ and further a 
factor of degree in $\{\frac{4(n-1)}{2} +1, \frac{4(n-1)}{2} +2\}$ when 
$n$ is odd. Suppose this is not true. By Corollary \ref{dlk}, we may assume 
that $P(\prod^{k-1}_{i=0}(\al +4(u+n-i))<4(u+k+1)$. Then by 
Theorem \ref{>4k+1} and Corollary \ref{>4k}, we obtain
$u=-1, k=2, \al +4(u+n-k+1)\in \{21, 45\}$ or $u=0,  k=2, \al +4(n-k+1)\in \{11, 21, 45, 77, 121\}$ or 
$u=0, k=3, \al +4(u+n-k+1)=117$. For the values of $u, k, n, \al$ given by these values, we obtain from Lemma \ref{mainupdate} that $G_{u+\frac{\al}{4}}(x^4)$ do not have a factor of degree in $\{4k, 4k-1, 4k-2, 4k-3\}$. 
When $u=-1$, we have $k=2$ and $(n, \al)\in \{(7, 1), (13, 1)\}$ and in both these cases, the prime 
$p=7$ works in Lemma \ref{mainupdate}. For $u=0,  k=2$, we have $(n, \al)=(3, 3)$ or $\al=1, n\in \{6, 12, 20, 31\}$. 
Since $n\ge 2k$, we have $\al=1$ and $n\in \{6, 12, 20, 31\}$ and prime $p=7$ works for $n\in \{6, 12, 20\}$ 
and $p=11$ works for $n=31$ in Lemma \ref{mainupdate}. For $u=0, k=3$, we have $\al=1, n=31$ and here the prime 
$p=11$ works in Lemma \ref{mainupdate}. 
\qed

\section{Proof of Theorems \ref{1/3} and \ref{1/4}}

We observe that if $G(x^d)$  has no factor of degree $\ge l$ with $l\le \frac{dn}{2}$,
then $G(x)$ has no factor of degree $\ge \frac{l}{d}$. Recall that by a factor, we meant the factor of degree
less than or equal to half of total degree and its co factor is the one whose degree is more than half of the
total degree. If $G(x^d)$ has a factor of degree $d$ only, then $G(x)$ may have a linear factor but no other
factor of degree $\ge 2$. Further if $G_{\frac{\al}{3}}(x^3)$ has a quadratic factor only or a factor of degree $5$
only, then  $G_{\frac{\al}{3}}(x)$
will be irreducible. Hence if the assertion of Theorems \ref{1/3} and \ref{1/4} are proved for  $G(x^d)$, then
the assertion of Theorems \ref{1/3} and \ref{1/4} follow.

Therefore we prove the assertions of Theorems \ref{1/3} and \ref{1/4} for  $G(x^d)$.
From Sections $5$ and $6$, we may assume that $G(x^d)$ has a factor of degree in
$\{1, \ldots, d\}$ except when $q=-\frac{1}{3}, n=43$ where $G_q(x^3)$ may have a factor of
degree $5$ and $q=\frac{2}{3}, n=42$ where $G_q(x^3)$ may have a factor of degree in $\{4, 5, 6\}$.
Then by Lemma \ref{mainupdate}, we may suppose that prime divisors of $\al+d(u+n)$ are given by
\begin{center}
\begin{tabular}{|c|c|c|c||c|c|c|c|} \hline
$d$ & $u$ & $\al$ & $p|\al+d(u+n)$ & $d$ & $u$ & $\al$ & $p|\al+d(u+n)$ \\ \hline
$3$ & $-1$ & $1$ & $2$ & $4$ & $-1$ & $1$ & $3$ \\ \hline
$3$ & $-1$ & $2$ & $2$ & $4$ & $-1$ & $3$ & $3$ \\ \hline
$3$ & $0$ & $1$ & $2$ & $4$ & $0$ & $1$ & $3, 5$  \\ \hline
$3$ & $0$ & $2$ & $2, 5$ & $4$ & $0$ & $3$ & $3, 7$  \\ \hline
\end{tabular}
\end{center}

\subsection{Proof of Theorem \ref{1/4}:}
Let $d=4$. We take $p$ to be the smallest prime dividing $\al+4(u+n)$. Thus $p=3$ unless
$\al+4(u+n)=1+4n=5^b$ for some positive integer $b$ where we take $p=5$ and $\al+4(u+n)=3+4n=7^c$
for some positive integer $c$ where we take $p=7$. We use Corollary \ref{1/k}. Taking
$m=4n, k\in \{1, 2, 3, 4\}, l=k-1$, we observe that the conditions of Corollary \ref{1/k}
are satisfied. We follow the notations as in the proof of Lemma \ref{mainupdate}. Let
\begin{align*}
\Delta_j=(\al +(u+1)d)\cdots (\al +(u+j)d).
\end{align*}
We show that
\begin{align}\label{<=1}
\phi_j=\frac{\nu_p(\Delta_j)}{j}\le 1 \ \ {\rm for} \ 1\le j\le n
\end{align}
and
\begin{align}\label{<1}
\phi_j<1 \ \ {\rm for} \ 1\le j\le n \ \ {\rm when} \ p=3, (u, \al)\in \{(-1, 1), (0, 3)\}.
\end{align}
This with Corollary \ref{1/k} with $p=5$ and $p=7$ according as $(u,\alpha)=(0, 1)$ and
$(u,\alpha)=(0, 3)$ respectively and $p=3$ if $u=-1$ will imply Theorem \ref{1/4}.

We follow as in the proof of Lemma \ref{mainupdate}. We have $j_0, l_0$ given by
\begin{center}
\begin{tabular}{|c|c|c|c|c||c|c|c|c|c|} \hline
$u$ & $\al$ & $p$ & $j_0$ & $l_0$ & $u$ & $\al$ & $p$ & $j_0$ & $l_0$ \\ \hline
$-1$ & $1$ & $3$ & $3$ & $3$ & $-1$ & $3$ & $3$ & $1$ & $1$  \\ \hline
$0$ & $1$ & $3$ & $2$ & $3$ & $0$ & $3$ & $3$ & $3$ & $5$  \\ \hline
$0$ & $1$ & $5$ & $1$ & $1$ & $0$ & $3$ & $7$ & $1$ & $1$  \\ \hline
\end{tabular}
\end{center}
We find that \eqref{<=1} and \eqref{<1} are valid for $1\le j\le 3$. Let $j>3$ and
we now show that $\phi_j<1$ for $j>3$. We can restrict to $j$ such that $p|(\al+4(u+j))$ and
such $j$ are given by $j=j_0+ps$ with $s>0$. As in the proof of Lemma \ref{mainupdate}, it suffices to show
\begin{align*}
\nu_p(\Delta _j)=s+1+\nu_p(l_0(l_0+4)\cdots (l_0+4s))<j_0+ps.
\end{align*}
This is true for $1\le s\le 3$. For $s\ge 4$, we find that the left hand side of the above inequality
is at most $s+1+\nu_p((l_0+4s)!)-1$ since there is at least one multiple of $p$ dividing $(l_0+4s)!$ but
not dividing $l_0(l_0+4)\cdots (l_0+4s)$. This together with $\nu_p(r!)<\frac{r}{p-1}$, $p\ge 3$ and $\frac{l_0}{2}<j_0$ imply
\begin{align*}
\nu_p(\Delta _j)\le s+\frac{l_0+4s}{p-1}\le s+\frac{l_0+4s}{2}=\frac{l_0}{2}+3s<j_0+ps.
\end{align*}
\qed

\subsection{Proof of Theorem \ref{1/3}:}
Let $d=3$. First assume that $u=0, \al=2$ and $5|(2+3n)$. We consider the polynomial $G_{\frac{2}{3}}(x^3)$.
We use Corollary \ref{1/k} with $p=5$ to show that $G_{\frac{2}{3}}(x^3)$ does not have a factor with degree in
$\{1, 2\}$. As in the proof of Lemma \ref{mainupdate}, it suffices to show
\begin{align*}
\nu_5(5\cdot 8\cdots (2+3j))<\frac{3j}{2}
\end{align*}
where we may assume that $j>1$. We obtain by using Lemma \ref{m!}  that
\begin{align*}
\nu_5(5\cdot 8\cdots (2+3j))\le \nu_5((2+3j)!)\le \frac{1+3j}{4}<\frac{3j}{2}.
\end{align*}
Hence $G_{\frac{2}{3}}(x^3)$ does not have a factor with degree in $\{1, 2\}$ in this case.

From now on, we may suppose that $5\nmid (2+3n)$ when $u=0, \al=2$. Therefore for each $u\in \{0, -1\}$ and
for each $\al\in \{1, 2\}$, we have $\al+3(u+n)=2^a$ for some integer $a>1$. We take $p=2$ and $\nu =\nu_2$
from now onwards in this section. We may assume by Section 5 that $G(x^3)$ has a factor of degree in $\{1, 2, 3\}$.
Let $\eta=0$ if $\al=1$ and $1$ if $\al=2$. From $\al+3(u+n)=2^a$, we have $a=2s+\eta$ for some $s>0$ and
$n=-u+2^\eta(1+2^2+\cdots +2^{2(s-1)})$. Put $n_0=0, n_s=n$ and
\begin{align}\label{n_i}
n_i=2^\eta(2^{2(s-1)}+2^{2(s-2)}+\cdots +2^{2(s-i)}) \ {\rm for} \ 1\le i\le s-1.
\end{align}
Then for $1\le i\le s-1$, we have
\begin{align*}
n_i-1=2^\eta(2^{2(s-1)}+2^{2(s-2)}+\cdots +2^{2(s-i+1)})+\sum^{2(s-i)+\eta-1}_{j=0}2^j
\end{align*}
and hence by Lemma \ref{m!}, we have
\begin{align}\label{ni}
\nu((n_i-1)!)=n_i-1-(i-1+2(s-i)+\eta)=n_i-a+i.
\end{align}
Also
\begin{align}\label{n}
\nu((n-1)!)=\begin{cases}
n-s & \ {\rm if} \ u=0, \al=1\\
n-s-1 & \ {\rm otherwise}.
\end{cases}
\end{align}
Let $1\le j<2^h$ for some $h>0$. Write $j-1=j_0+2j_1+\cdots +2^{h-1}j_{h-1}$ in base $2$
with $0\le j_u\le 1$ for $0\le u<h$. Note that $\sum^{h-1}_{u=0}j_u\le h-1$.
 Hence by Lemma \ref{m!}, we have
\begin{align}\label{j}
\nu((j-1)!)=j-1-\sum^{h-1}_{u=0}j_u\ge j-1-(h-1)=j-h.
\end{align}

For $1\le i\leq n-1$, if $\al+3(u+i)=2^rt$ with $2\nmid t$, then from $3(n-i)=2^r(2^{a-r}-t)$, we obtain
$\nu(\al+3(u+i))=r=\nu(n-i)$. Therefore
\begin{align*}
\nu({\displaystyle{\prod_{i=l}^{n}}} (\al+3(u+i)))=a+\nu((n-l)!) \ {\rm for} \ 1\le l<n-1.
\end{align*}
We now consider the $G(x^3)$ with all $a_j'$s equal to $1$ and call it $G^*$. Recall that 
$G=G_{u+\frac{\al}{3}}$ with $(u, \al)\in \{(-1, 1), (-1, 2), (0, 1), (0, 2)\}$. Then the Newton Polygon 
$NP_2(G^*)$ of $G^*$ with respect to prime $2$ is given by the lower edges along the convex hull of the
following points
\begin{align*}
\{(0, 0), (3, a), (3\cdot 2, a), \cdots, (3l, a+\nu((l-1)!)), \cdots,  (3n, a+\nu((n-1)!))\}
\end{align*}
in the extended plane. Let $a=2s+\eta \le 5$. Then $\al=1, (u, n)\in \{(-1, 2), (-1, 6), (0, 5)\}$ or
$\al=2, (u, n)\in \{(-1, 3), (0, 2), (-1, 11), (0, 10)\}$.
For these values of $(\al, u, n)$, we check that assertion of the Theorem \ref{1/3}
holds by using Lemma \ref{<r}. For example, when $(\al, u, n)=(2, -1, 11)$,
we find that the breaks of $NP_2(G^*)$ are given by $0<3\cdot 8<3\cdot 11$ and the minimum slope
is $\frac{3}{8}$ and the maximum slope is $\frac{4}{9}$. For $t\in \{1, 2, 3\}$, taking
$r=\lf\frac{t}{3}\rf$ in Lemma \ref{<r}, we obtain that $G_{\frac{-2}{3}}(x^3)$ does not have a factor
of degree $t$ and hence irreducible. Similarly we use Lemma \ref{<r} to get the assertion of Theorem \ref{1/3}
in the remaining cases.

Hence from now on, we assume that $a\ge 6$.
If $(0, 0)$ and $(3n, a+\nu((n-1)!))$ are the only lattice points on the Newton Polygon
$NP_2(G^*)$, then from \eqref{n}, the unique slope is
\begin{align*}
\frac{a+\nu((n-1)!)}{3n}\le \frac{2s+\eta +n-s}{3n}=\frac{1}{3}+\frac{2s+2\eta}{2\cdot 3n}=
\frac{1}{3}+\frac{a+\eta}{2\cdot 3n}\le \frac{5}{4\cdot 3}
\end{align*}
since $n\ge \frac{2^a-2}{3}\ge 2(a+1)\ge 2(a+\eta)$ for $a\ge 6$. Also the unique slope is
$>\frac{1}{3}$. Then by using Lemma \ref{<r} for $t\in \{1, 2, 3\}$ with $r=\lf\frac{t}{3}\rf$, we obtain
$G(x^3)$ is irreducible. Hence we may suppose that there is a lattice point of $NP_2(G^*)$ with $x$
co-ordinate lying in $(0, 3n)$. We prove that the breaks of $NP_2(G^*)$ are given by
$0=3n_0<3n_1<3n_2<\cdots <3n_{s-2}<3n_s=3n$ if
$(u, \al)=(-1, 1)$ and $0=3n_0<3n_1<3n_2<\cdots <3n_{s-1}<3n_s=3n$
otherwise.

First we show that $(3n_1, a+\nu((n_1-1)!))$ is a lattice point on $NP_2(G^*)$. It suffices to show
\begin{enumerate}
\item[$(i)$] $\frac{a+\nu((i-1)!)}{i}> \frac{a+\nu((n_1-1)!)}{n_1}$ for $1\le i<n_1$.
\item[$(ii)$] $\frac{a+\nu((n_l-1)!)}{n_l}>\frac{a+\nu((n_1-1)!)}{n_1}$ for $2\le l<s$.
\item[$(iii)$] $\frac{a+\nu((i-1)!)}{i}>\frac{a+\nu((n_1-l)!)}{n_1}$ for $n_l<i<n_{l+1}, 1\le l<s$.
\end{enumerate}

$(i):$ Let $1\le i<n_1=2^{a-2}$. Then from \eqref{j} and \eqref{ni},
\begin{align*}
&n_1\{a+\nu((i-1)!)\}-i\{a+\nu((n_1-1)!)\}\\
&\ge n_1\{a+i-a+2\}-i\{a+n_1-a+1\}=2n_1-i>0.
\end{align*}

$(ii):$ For $2\le l<s$, we have from \eqref{j}
\begin{align*}
\frac{a+\nu((n_l-1)!)}{n_l}-\frac{a+\nu((n_1-1)!)}{n_1}=\frac{n_l+l}{n_l}-\frac{n_1+1}{n_1}=
\frac{l}{n_l}-\frac{1}{n_1}>0
\end{align*}
since $n_l=2^\eta(2^{2(s-1)}+2^{2(s-2)}+\cdots +2^{2(s-l)})<l2^{\eta+2(s-1)}=ln_1$.

$(iii):$ Let $1\le l<s$. Write $i=n_l+j$ with $1\le j<n_{l+1}-n_l=2^{a-2l-2}$. Since $\nu(u)=\nu(n_l+u)$
for any $1\le u<n_{l+1}-n_l$, we get $\nu((i-1)!)=\nu((n_l-1)!)+\nu(n_l)+\nu((j-1)!)$. This
with \eqref{n_i}, \eqref{ni} and \eqref{j} imply
\begin{align*}
&\frac{a+\nu((i-1)!)}{i}-\frac{a+\nu((n_1-1)!)}{n_1}\ge \frac{n_l+l+j+2}{n_l+j}-\frac{n_1+1}{n_1}\\
&=\frac{1}{n_1(n_l+j)}((l+2)n_1-n_l-j))> \frac{1}{n_1(n_l+j)}\{(l+2)n_1-n_{l+1}\}>0
\end{align*}
since $n_{l+1}=2^\eta(2^{2(s-1)}+\cdots +2^{2(s-l)}+2^{2(s-l-1)})<(l+1)2^{2s+\eta-2}<(l+2)n_1$.
Hence the minimum slope is $\frac{1}{3}(1+\frac{1}{n_1})$.

Let $1\le l<s-2$. Next we show that if $(3n_l, a+\nu((n_l-1)!))$ is a lattice point on $NP_2(G^*)$, then
$(3n_{l+1}, a+\nu((n_{l+1}-1)!))$ is a lattice point on $NP_2(G^*)$. Assume that
$(3n_l, a+\nu((n_l-1)!))$ is a point on $NP_2(G^*)$. If $(3n, a+\nu((n-1)!))$ is the next
lattice point, then from \eqref{n_i}-\eqref{n}, we see that slope of the rightmost edge is
\begin{align*}
\frac{\nu((n-1)!)-\nu((n_l-1)!)}{3(n-n_l)}\le \frac{n-s-(n_l-a+l)}{3(n-n_l)}\le
\frac{1}{3}+\frac{s+\eta -l}{3(n-n_l)}\le \frac{5}{4\cdot 3}
\end{align*}
since $1\le l<s-2$ and $n-n_l\ge 2^{\eta}\frac{2^{2(s-l)}-1}{3}\ge 4(\eta+s-l)$ for $s-l\ge 3$.
Observe that $n_1>3$ and the slope of the leftmost edge is $\frac{1}{3}(1+\frac{1}{n_1})$.
We now apply Lemma \ref{<r} for $t\in \{1, 2, 3\}$ with $r=\lf\frac{t}{3}\rf$ to obtain
$G(x^3)$ is irreducible. Thus we may suppose that $(3n, a+\nu((n-1)!))$ is not the next lattice
point on $NP_2(G^*)$. To show $(3n_{l+1}, a+\nu((n_{l+1}-1)!))$ is the next lattice point on $NP_2(G^*)$, it
suffices to show
\begin{enumerate}
\item[$(iv)$] $\frac{\nu((n_u-1)!)-\nu((n_l-1)!)}{n_u-n_l}>\frac{\nu((n_{l+1}-1)!)-\nu((n_l-1)!)}{n_{l+1}-n_l}$ for $l+1<u\le s$.
\item[$(v)$] $\frac{\nu((i-1)!)-\nu((n_l-1)!)}{i-n_l}>\frac{\nu((n_{l+1}-1)!)-\nu((n_l-1)!)}{n_{l+1}-n_l}$ for $n_u<i<n_{u+1}, l\le u<s$.
\end{enumerate}
The assertion $(iv)$ follows from \eqref{ni} and by observing $(u-l)2^{2(s-l-1)+\eta}> 2^\eta(2^{2(s-l-1)}+\cdots +2^{2(s-u)})$. The assertion $(v)$ follows like $(iii)$ above by
observing that if $i=n_u+j$ with $1\le j<n_{u+1}-n_u=2^{a-2u-2}$ and
$(u-l+2)2^{2(s-l-1)+\eta}>2^\eta(2^{2(s-l-1)}+\cdots +2^{2(s-u-1)})=n_{u+1}-n_l\ge n_u+j-n_l$.

Thus we need to check for lattice points after $(3n_{s-2}, a+\nu((n_{s-2}-1)!))$ on $NP_2(G^*)$.
Recall that $n_{s-2}=n+u-2^\eta-2^{2+\eta}$ and $\nu(n-i)=\nu(\al+3(u+i))$ for $i\ge 1$.
For $(u, \al)=(-1, 1)$, we find that $n_{s-2}=n-6$ and check using $\nu(n-i)=\nu(\al+3(u+i))$ for
$i\ge 1$ that $(3n, a+\nu((n-1)!)))$ is the lattice point after $(3n_{s-2}, a+\nu((n_{s-2}-1)!))$ and
hence the maximum slope is $\frac{7}{18}$.
For $(u, \al)=(-1, 2)$, we find that $n_{s-1}=n-3$ and $(3(n-3), a+\nu((n-3)!))$ and
$(3n, a+\nu((n-1)!)))$ are the lattice points after $(3n_{s-2}, a+\nu((n_{s-2}-1)!))$
and the maximum slope is $\frac{4}{9}$. For $(u, \al)=(0, 2)$, we find that
$n_{s-1}=n-2$ and $(3(n-2), a+\nu((n-3)!))$ and $(3n, a+\nu((n-1)!)))$ are the lattice points after
$(3n_{s-2}, a+\nu((n_{s-2}-1)!))$ and the maximum slope is $\frac{1}{2}$.
For $(u, \al)=(0, 1)$, we find that $n_{s-1}=n-1$ and $(3(n-1), a+\nu((n-3)!))$ and $(3n, a+\nu((n-1)!)))$
are the lattice points after $(3n_{s-2}, a+\nu((n_{s-2}-1)!))$ and the maximum slope is $\frac{2}{3}$.
Recall that in all these cases, the slope of the leftmost edge is $\frac{1}{3}(1+\frac{1}{n_1})$.

We now use Lemma \ref{<r} for $t\in \{1, 2, 3\}$ with $r=\lf\frac{t}{3}\rf$ to obtain
that $G_{\frac{-1}{3}}(x^3)$ and $G_{\frac{-2}{3}}(x^3)$ are irreducible. Further
$G_{\frac{1}{3}}(x^3)$ does not have a factor of degree $1$ and $G_{\frac{2}{3}}(x^3)$ do not have a
factor of degree $1$ or $3$.
\qed

\section{Proof of Theorem \ref{Lag1/3}}

We first check that $L^{(\frac{1}{4})}_2(x)$ and $L^{(\frac{1}{4})}_2(x^4)$ are not irreducible
and their factorizations are given in the statement of Theorem \ref{Lag1/3}. Therefore we assume from
now on that $n\neq 2$ when $q=\frac{1}{4}$. We observe that the irreducibility of $L^{(q)}_n(x^d)$
implies the irreducibility of $L^{(q)}_n(x)$. Hence we show that $L^{(q)}_n(x^d)$ is irreducible. For
$(q, n)\in \{(-\frac{2}{3}, 2), (-\frac{1}{3}, 43), (\frac{2}{3}, 42)\}$, we check that  $L^{(q)}_n(x^3)$ are
irreducible. Thus from Theorems \ref{1/3} and \ref{1/4}, we need to consider only the following cases:
\begin{align}\label{qn}
\begin{split}
q=\frac{1}{3},  \ &1+3n=2^a \\
q=\frac{2}{3},  \ &2+3n=2^a5^b, a\ge 0, b\ge 0\\
q=-\frac{1}{4}, \ &3+4(n-1)=3^a\\
q=\frac{1}{4}, \ &1+4n=3^a5^b, a\ge 0, b\ge 0\\
q=\frac{3}{4}, \ &3+4n=7^a
\end{split}
\end{align}
Further it  suffices to show that $n!L^{(q)}_n(x^d)$ does not have a factor of degree $d$ and 
for $q\in \{\frac{1}{3}, \frac{2}{3}\}$, $n!L^{(q)}_n(x^3)$ do not have a quadratic or a cubic factor.
In fact we show that it does not have a factor of degree $\le d$. First we prove

\begin{lemma}\label{pdivn}
For $n>1$ given by \eqref{qn}, there is a prime $p|n$ such that
\begin{align*}
p\nmid d(\al+(u-1)d)(\al+ud)(\al+(u+1)d)
\end{align*}
except when $q=\frac{2}{3},  n\in \{2, 6, 10, 16\}$ and $q=\frac{1}{4},  n\in \{6, 20\}$.
\end{lemma}

\begin{proof}
Let $n>1$ be given by \eqref{qn}. Suppose that $p|n$ implies $p|d(\al+(u-1)d)(\al+ud)(\al+(u+1)d)$.

Let $q=\frac{1}{3}$. Then $p|n$ implies $p\in \{2, 3\}$. Writing $n=2^r3^s$, we have
$2^a=1+3n=1+2^r3^{1+s}$ implying $r=0$, $2^a-3^{1+s}=1$. By Lemma
\ref{nag}, we have $2^2-3=1$ or $2^a=4$ and $3^{1+s}=3$  giving $n=1$ which is not possible.

Let $q=\frac{2}{3}$. Then $p|n$ implies $p\in \{2, 3, 5\}$. Writing $n=2^r3^s5^t$, we have
$2^a5^b=2+3n=2+2^r3^{1+s}5^t$. If $a=0$, then $r=t=0$ and $5^b=2+3^{1+s}$. By Lemma
\ref{nag}, we have $5=2+3$ giving $n=1$ which is not possible. Hence $a\neq 0$. If $b=0$,
then $a>1$ giving $r=1$, $2^a=2+2\cdot 3^{1+s}5^t$ or $2^{a-1}=1+3^{1+s}5^t$. By Lemma \ref{nag},
we get solutions $2^2=1+3$ and $2^4=1+3\cdot 5$ giving $n\in \{2, 10\}$. Hence assume that
$ab\neq 0$. Then $t=0$ and $2^a5^b=2+2^r 3^{1+s}$. If $a=1$, then $2\cdot 5^b=2+2^r3^{1+s}$ or
$5^b=1+2^{r-1}3^{1+s}$. By Lemma  \ref{nag}, the solution $5^2=1+2^3\cdot 3$ gives $n=16$. Finally let $a>1$.
Then $u=1$ and we get $2^a5^b=2+2\cdot 3^{1+s}$ or $2^{a-1}5^b=1+3^{1+s}$. By Lemma \ref{nag}, its
solution $2\cdot 5=1+3^2$ gives $n=6$.

Let $q=-\frac{1}{4}$. Then $p|n$ implies $p\in \{2, 3, 5\}$. Writing $n=2^r3^s5^t$, we have
$3^a=4n-1=2^{2+r}3^s5^t-1$ implying $v=0$ and $2^{2+r}5^t-3^a=1$. By
Lemma \ref{nag}, its solution is $2^2-3^1=1$ which gives $n=1$. This is not possible.

Let $q=\frac{1}{4}$. Then $p|n$ implies $p\in \{2, 3, 5\}$. Writing $n=2^r3^s5^t$, we have
$3^a5^b=1+4n=1+2^{2+r}3^s5^t$. Let $a=0$. Then $t=0$ and $5^b=1+2^{2+r}3^s$ and by Lemma
\ref{nag}, its solutions $5=1+2^2$ and $5^2=1+2^3\cdot 3$ give $n=6$ since $n>1$. Let $b=0$. Then
$s=0, 3^a=1+2^{2+r}5^t$ and by Lemma \ref{nag}, its solutions $3^2=1+2^3$ and
$3^4=1+2^4\cdot 5$ give $n=20$ since $n\neq 2$. Finally let $ab\neq 0$. Then $s=t=0, 3^a5^b=1+2^{2+r}$
and by Lemma \ref{nag}, there are no solutions.

Let $q=\frac{3}{4}$. Then $p|n$ implies $p\in \{2, 3, 7\}$. Writing $n=2^r3^s7^t$, we have
$7^a=3+4n=3+2^{2+r}3^s7^t$ implying $s=t=0$ and $7^a=3+2^{2+r}$. By Lemma \ref{nag}, its solution
$7=3+2^2$ imply $n=1$ which is not possible.
\end{proof}

For $n\in \{2, 6, 10, 16\}$ if $q=\frac{2}{3}$ and $n\in \{6, 20\}$ if $q=\frac{1}{4}$, we check that
$L^{(q)}_n(x^d)$ are irreducible. Hence we may suppose that
$n\notin \{2, 6, 10, 16\}$ if $q=\frac{2}{3}$ and $n\notin \{2, 6, 20\}$ if $q=\frac{1}{4}$. Then by
Lemma \ref{pdivn},  we find that there is a prime $p|n$ such that
$p\nmid d(\al+(u-1)d)(\al+ud)(\al+(u+1)d)$. Let $p$ be largest with this property.
Thus we always have $p\ge 5>d$. 
We use Corollary \ref{1/k} with $k=d, l=0$. Since $p|\binom{n}{j}$ for $1\le j<p$ and
$p|\prod^p_{i=1}(\al+(u+i)d)$, the conditions of  Corollary \ref{1/k} are satisfied. It suffices to show
\begin{align*}
\nu_p\left(\prod^j_{i=0}(\al+(u+i)d)\right)-\nu_p\left(\binom{n}{j}\right)<\frac{dj}{d}=j \ \ {\rm for} \ 1\le j\le n.
\end{align*}
Observe that $p$ divides at most one of $\al+(u+i)d$ when $1\le i<p$ and $\al+(u+p-1)d<pd<p^2$. By using
$p|\binom{n}{j}$ for $1\le j<p$, we obtain that the left hand side of above inequality is $\le 0$ for $1\le j<p$ and
hence the assertion follows for $1\le j<p$. Let $j\ge p$. Then there is at least one multiple of $p$ dividing
$(\al+(u+j)d)!$ but not dividing $\prod^j_{i=0}(\al+(u+i)d)$. Therefore by using Lemma \ref{m!}, we obtain
\begin{align*}
\nu_p\left(\prod^j_{i=0}(\al+(u+i)d)\right)-\nu_p\left(\binom{n}{j}\right)&\le \nu_p((\al+(u+j)d)!)-1\\
&\le \frac{\al+(u+j)d-1}{p-1}-1 \le u+j+\frac{\al-1}{p-1}-1<j
\end{align*}
by using Lemma \ref{m!} and since $p>d>\al$.
\qed

\section{Acknowledgments}

We would like to sincerely thank the referee for careful reading of our manuscript and for pointing out inaccuracies in an earlier draft of this paper and making suggestions in the exposition. 
The first author was supported in part by the DST Project SR/FTP/MS-035/2012.
The second author would like to thank the Max-Planck Institute of Mathematics, Bonn for
its hospitality where a part of this work was done in his visit during May-July 2014.

\end{document}